\documentclass[12pt]{amsart}
\title[Multiplicity-free representations over Siegel domains]
{Multiplicity-free representations of certain
\\ nilpotent Lie groups over 
\\ Siegel domains of the second kind}
\author{Koichi Arashi}
\address{K. Arashi: Department of Mathematics, Tokyo Gakugei University, Nukuikita 4-1-1, Koganei,Tokyo 184-8501, Japan}
\email{arashi@u-gakugei.ac.jp}

\usepackage{authblk}
\usepackage{hyperref}
\usepackage{color}
\usepackage{enumitem}
\numberwithin{equation}{section}
\theoremstyle{definition}
\newtheorem{example}{Example}[section]
\newtheorem{definition}[example]{Definition}
\newtheorem{remark}[example]{Remark}
\theoremstyle{plain}
\newtheorem{lemma}[example]{Lemma}
\newtheorem{theorem}[example]{Theorem}
\newtheorem{proposition}[example]{Proposition}
\newtheorem{corollary}[example]{Corollary}

\newcommand{\U}{\mathbb{\mathbb{C}}^M}
\newcommand{\V}{{\mathbb{\mathbb{R}}^N}}
\newcommand{\cc}{s}
\newcommand{\iu}{i}
\newcommand{\lm}{\lambda}
\newcommand{\reference}{p}
\newcommand{\ua}{\"a}
\DeclareMathOperator{\Ad}{Ad}

\begin{document}
\setlength{\abovedisplayskip}{4.81pt}
\setlength{\belowdisplayskip}{5pt}
\date{}
\maketitle
\begin{abstract}
We investigate the multiplicity-freeness property for the holomorphic multiplier representations of affine transformation groups of a Siegel domain of the second kind.
We deal with the generalized Heisenberg group and its subgroups. Necessary and sufficient conditions for a specific representation to be multiplicity-free are provided.
We study the multiplicity-freeness property in relation to the geometrical notions of coisotropic action and visible action, and also the commutativity of the algebra of invariant differential operators.
\bigskip

\end{abstract}
\section{Introduction}

In this article, we study multiplicity-free representations of certain nilpotent affine transformation groups of Siegel domains of the second kind.
This class of domains includes (up to holomorphic equivalence) non-tube type bounded homogeneous domains, and as special cases, non-tube type bounded symmetric domains, and we regard these domains as K\ua hler manifolds by the Bergman metric.

Let $N$ and $M$ be positive integers.
For a regular cone $\Omega\neq\emptyset\subset\V$ and an $\Omega$-positive Hermitian map $Q:\U\times\U\rightarrow \mathbb{C}^N$, put
\begin{equation*}
\mathcal{D}(\Omega,Q):=\{(z,u)\in\mathbb{C}^N\times \U|\, \mathop{\mathrm{Im}}z-Q(u,u)\in\Omega\},
\end{equation*}
which is called a Siegel domain of the second kind.
For $x_0\in \V$ and $u_0\in\U$, let us denote by $n(x_0, u_0)$ the affine transformation
\begin{equation*}
\mathbb{C}^N\times\U\ni(z,u)\mapsto (z+x_0+2\iu Q(u,u_0)+\iu Q(u_0,u_0),u+u_0)\in\mathbb{C}^N\times\U.
\end{equation*}
Let $G:=\{n(x_0,u_0)\mid x_0\in\mathbb{R}, u_0\in\U\}$.
The maps $n(x_0,u_0)\,(x_0\in\V, u_0\in \U)$ preserve $\mathcal{D}(\Omega,Q)$, and hence $G$ acts on $\mathcal{D}(\Omega,Q)$.
We prove the following theorem.

\begin{theorem}[see Theorems \ref{Z8CrYFczzHas}, \ref{tyjqndhg5v55}]\label{p9UgSPw2}
For any $G$-equivariant holomorphic line bundle $L$ over $\mathcal{D}(\Omega,Q)$, the natural representation of $G$ on the space $\Gamma^{hol}(\mathcal{D}(\Omega,Q),L)$ of holomorphic sections of $L$ is multiplicity-free.
\end{theorem}
Let $\pi_0$ be the representation of $G$ on the space $\mathcal{O}(\mathcal{D}(\Omega,Q))$ of holomorphic functions on $\mathcal{D}(\Omega,Q)$ defined by
\begin{equation*}
\pi_0(g)f(z,u):=f(g^{-1}(z,u))\quad(g\in G, f\in \mathcal{O}(\mathcal{D}(\Omega,Q)), (z,u)\in\mathcal{D}(\Omega,Q)).
\end{equation*}
The representation $\pi_0$ is a simple example of the representation of $G$ in Theorem \ref{p9UgSPw2}, and at least in this case, we can regard the multiplicity-freeness in a notably strong sense that any unitary representation of $G$ realized in $\mathcal{O}(\mathcal{D}(\Omega,Q))$ is multiplicity-free (see Definitions \ref{def:invariant hilbert}, \ref{def:mfu} and also Remark \ref{rem:mfc}).
We shall make some comments on related work, which motivated our study of these representations of $G$.
Gindikin \cite{Gindikin} gives an integral expression for the Bergman kernel of $\mathcal{D}(\Omega,Q)$ (see Theorem \ref{K3tmLfdV}), from which we can see a multiplicity-free direct integral decomposition of the unitary subrepresentation of $\pi_0$ on the space of holomorphic $L^2$-functions on $\mathcal{D}(\Omega,Q)$.
When $\mathcal{D}(\Omega,Q)$ is homogeneous, Ishi \cite[Section 4]{Ishi99} gives the direct integral decompositions of unitary subrepresentations of $\pi_0$, which are the restrictions of representations of a maximal connected real split solvable Lie subgroup of the holomorphic automorphism group of $\mathcal{D}(\Omega,Q)$.

To show Theorem \ref{p9UgSPw2}, we first classify all $G$-equivariant holomorphic line bundles over $\mathcal{D}(\Omega,Q)$.
Next, using a description of the multiplicities for the restriction of a unitary representation by Corwin and Greenleaf \cite{CG88}, we show that any infinite dimensional irreducible unitary representation of $G$ realized in $\Gamma^{hol}(\mathcal{D}(\Omega,Q),L)$ is a coherent state representation in the sense of Lisiecki \cite{Lisiecki95}.
Accordingly, such a representation has a vector, on which a certain complex Lie algebra acts by scalars.
We then complete the proof of Theorem \ref{p9UgSPw2} by showing such a vector in $\Gamma^{hol}(\mathcal{D}(\Omega,Q),L)$ is determined uniquely up to a constant.

Another motivation of our study of representations of $G$ is geometrical aspects of multiplicity-free representations.
For instance, coisotropic action \cite{HW90}, spherical variety \cite{VK79}, and visible action \cite{Kobayashi05, Kobayashi13} are distinct notions that provide machinery for generating multiplicity-free representations.
Not limited to the case of $G$ only, we study certain subgroups of $G$.
Let $W\subset \U$ be a real subspace, and put
\begin{equation*}
G^W:=\{n(x,u)\mid x\in \V,u\in W\}.
\end{equation*}
We first address the problem of the existence of the strongly visible action (Definition \ref{b5SgLeVG}) of $G^W$.
\begin{theorem}[see Theorem \ref{th:real form2}]\label{th:real form}
The action of $G^W$ on $\mathcal{D}(\Omega,Q)$ is strongly visible with respect to an involutive anti-holomorphic diffeomorphism if and only if $W$ contains a real form $W_0$ of $\U$ such that
\begin{equation}\label{eq:imq}
\displaystyle\mathop{\mathrm{Im}} Q(W_0,W_0)=0.
\end{equation}
\end{theorem}
We note that, by the multiplicity-free theorem \cite[Corollary 2.3]{Kobayashi13}, the natural representation $ \pi_0|_{G^W}$ is multiplicity-free, assuming the visibility condition, which is a key step in the proof of Theorem \ref{th:real form}.

Concerning the strongly visible action of a non-reductive Lie group, Kobayashi pointed out in \cite{Kobayashi07} that for a Hermitian symmetric space $H/K$ of noncompact type, the action of a maximal unipotent subgroup of $H$ is strongly visible.
Note that our group $G^W$ is unipotent, though it is not, in general, a maximal unipotent subgroup of $H$ in this case.
For the strongly visible action of an affine automorphism group containing $G$ as a subgroup, see \cite{Arashi}.
In contrast, we deal with smaller groups, and Theorem \ref{th:real form} reveals a nontrivial constraint concerning the strongly visible action.
Note that recently, several authors establish equivalences between the strong visibility and various multiplicity-free conditions.
In \cite[Theorem 1.2]{Tanaka22}, Tanaka proved that for a connected complex reductive group $H$ and a spherical $H$-variety $D$, the action of a compact real form of $H$ is strongly visible.
For a non-reductive group, Baklouti and Sasaki \cite{BS21} studied the quasi-regular representations of the Heisenberg group, and related the multiplicity-freeness of such a representation to the strongly visibility of a group action on a complex Heisenberg homogeneous space.

The condition \eqref{eq:imq} can be regarded as a geometrical constraint on $\mathcal{D}(\Omega,Q)$ (see \cite[Proposition 1.3]{Gindikin}), and we have the following corollary.
\begin{corollary}\label{cor:tube}
The action of $G$ on $\mathcal{D}(\Omega,Q)$ is strongly visible with respect to an involutive anti-holomorphic diffeomorphism if and only $\mathcal{D}(\Omega,Q)$ is holomorphically equivalent to a tube domain.
\end{corollary}

Theorem \ref{p9UgSPw2} is a part of a study of the more general group $G^W$, which will be dealt with in another paper.
In this article, we focus on the special case that $\dim W=M$.
For a complex manifold $D$ and a Lie group $H$ acting holomorphically on $D$, let $\mathbb{D}_{H}(D)$ be the algebra of $H$-invariant differential operators on $D$ with holomorphic coefficients.
We prove the following theorem.
While the implications \ref{cond:mainvisible}$\Rightarrow$\ref{cond:mainmf}, \ref{cond:mainQ}$\Rightarrow$\ref{cond:mainvisible} of the theorem are included in the previous arguments about Theorem \ref{th:real form} and the proof of Theorem \ref{th:real form2}, the remaining parts (\ref{cond:mainmf}$\Leftrightarrow$\ref{cond:mainQ}$\Leftrightarrow$\ref{cond:maincois}$\Leftrightarrow$\ref{cond:maincomm}) constitute one of the main contributions of this article.
\begin{theorem}\label{X6mQfvEy}
For a real form $W\subset\U$, the following are equivalent.
\begin{enumerate}[label=\textup{(\roman*)}]
\item
The representation $\pi_0|_{G^W}$ is multiplicity-free;\label{cond:mainmf}
\item
$\mathop{\mathrm{Im}}Q(W,W)=0$;\label{cond:mainQ}
\item
Any $G^W$-orbit in $\mathcal{D}(\Omega,Q)$ is a coisotropic submanifold;\label{cond:maincois}
\item
$\mathbb{D}_{G^W}(\mathcal{D}(\Omega,Q))$ is commutative;\label{cond:maincomm}
\item
The action of $G^W$ on $\mathcal{D}(\Omega,Q)$ is strongly visible.\label{cond:mainvisible}
\end{enumerate}
\end{theorem}
We shall make some remarks.
As regards (iii), Huckleberry and Wurzbacher {\cite[Theorem 7]{HW90}} proved a multiplicity-free theorem for the representation of a compact Lie group on the space of holomorphic $L^2$ sections of an equivariant holomorphic line bundle over a K\ua hler manifold.
Concerning (iv), Faraut and Thomas \cite[Theorem 4]{FT99} gave, for a complex manifold $D$ and  a Lie group $H$ acting holomorphically on $D$, some sufficient conditions for $\mathbb{D}_H(D)$ to be commutative.

To prove (i) $\Rightarrow$ (ii), we use the aforementioned result by Corwin and Greenleaf.
For the proof of (ii) $\Rightarrow$ (iii), we first show that (ii) implies the visibility of the action of $G^W$.
Next we prove (iii), following the arguments for visible actions on K\ua hler manifolds in \cite[\S 4.3]{Kobayashi05}.
The integral expression of the Bergman kernel is also an important ingredient of the proof of (ii) $\Leftrightarrow$ (iii).

This article is organized as follows.
In Sect. \ref{Ke2QcWAT}, we fix some notations about Siegel domains and affine transformation groups, and review the notion of visible aciton.
Then we give a proof of Theorem \ref{th:real form}.
In Sect. \ref{Ee3FnBHw}, we classify isomorphism classes of $G$-equivariant holomorphic line bundles over $\mathcal{D}(\Omega,Q)$.
In Sect. \ref{u6EAbPUB} we prove Theorem \ref{p9UgSPw2} based on the classification obtained in Sect. \ref{Ee3FnBHw}.
In Sect. \ref{Fi4dG3gb}, we prove \ref{cond:mainmf}$\Leftrightarrow$\ref{cond:mainQ}$\Leftrightarrow$\ref{cond:maincois}$\Leftrightarrow$\ref{cond:maincomm} of Theorems \ref{X6mQfvEy}.

\section{Visible action on Siegel domain}\label{Ke2QcWAT}
The aim of this section is to provide a proof of Theorem \ref{th:real form}.
We deduce constraints on $Q$ from certain invariant functions, which also separate the group orbits in $\mathcal{D}(\Omega,Q)$.
We will do this by linearizing the problem.
Throughout this article, for a Lie group, we denote its Lie algebra by
the corresponding Fraktur small letter.
Also for a vector space or a Lie algebra $V$ over $\mathbb{R}$, we denote by $V_\mathbb{C}$ its complexification $V\otimes_\mathbb{R}\mathbb{C}$.
\subsection{Generalized Heisenberg group}

Let $\Omega\subset \V$ be a {\it regular cone}, that is, a nonempty open convex cone which contains no entire straight line.
Let $Q:\U\times \U\rightarrow \mathbb{C}^N$ be an {\it $\Omega$-positive} Hermitian map, i.e. a sesquilinear map satisfying $Q(u,u)\in\overline{\Omega}\backslash\{0\}\,(u\neq 0)$, where $\overline{\Omega}$ is the closure of $\Omega$.
The set $G$ is closed under composition, and has the natural structure of a Lie group, which we call the {\it generalized Heisenberg group}.
We identify the Lie algebra $\mathfrak{g}$ of $G$ with $\V\oplus\U$ so that $\exp(x,u)=n(x,u)\,(x\in \V, u\in \U)$ holds.
Then we have the equality
\begin{equation}\label{4im}
[u,v]=4\mathop{\mathrm{Im}}Q(u,v)\quad(u,v\in\U).
\end{equation}
Put $\mathcal{D}(\Omega):=\{z\in\mathbb{C}^N\mid \mathop{\mathrm{Im}}z\in\Omega\}$.

\subsection{Strongly visible action and constraints on $Q$}\label{t2AC4nXS}
We first recall the notion of visibility.
Suppose that a Lie group $H$ acts on a connected complex manifold $D$ by holomorphic automorphisms.

\begin{definition}[{\cite[\S 3]{Kobayashi05}}]\label{b5SgLeVG}
\begin{enumerate}
\item
We call the action {\it previsible} if there exists a totally real submanifold $S$ in $D$ and a (non-empty) $H$-invariant open subset $D'$ of $D$ such that $S$ meets every $H$-orbit in $D'$.
\item
A previsible action is called {\it visible} if $J_x(T_xS)\subset T_x(H\cdot x)$ for all $x\in S$, where $J_x\in \mathrm{End}(T_xD)$ denotes the complex structure.
\item
A previsible action is called {\it strongly visible} if there exists an anti-holomorphic diffeomorphism $\sigma$ of $D'$ such that
\begin{enumerate}
\item
$\sigma|_S=\mathrm{id}$,
\item
$\sigma \text{ preserves each }H\text{-orbit of }D'$.\label{cond:preserve}
\end{enumerate}\label{cond:sv}
\end{enumerate}
\end{definition}
In this article, our focus is on the case that $D=\mathcal{D}(\Omega,Q)$.
The standard complex structure on $\U$ will be denoted by $j$.
\begin{remark}\label{rem: (H)}
The condition \eqref{cond:preserve} is sufficient to ensure the multiplicity-freeness property for a particular representation, as discussed in \cite[Corollary 2.4]{Kobayashi13}, and \cite[p. 389]{FT99} when $\sigma$ is involutive.
\end{remark}

\begin{proposition}\label{Xn9bCyEc}
Suppose that the action of $G$ on $\mathcal{D}(\Omega,Q)$ is strongly visible.
Then there exists an antilinear map $A:\U\rightarrow \U$ such that
\begin{equation}\label{DFiWga8x}
Q(u,v)=Q(Av,Au)\quad(u,v\in \U).
\end{equation}
\end{proposition}

\begin{proof}
Let $\mathcal{D}'\neq\emptyset\subset \mathcal{D}(\Omega,Q)$ be a $G$-invariant open set and $\sigma:\mathcal{D}'\rightarrow\mathcal{D}'$ an anti-holomrophic diffeomorphism preserving all $G$-orbits in $\mathcal{D}'$.
For $(z,u)\in\mathcal{D}'$, we write $\sigma(z,u)=(\sigma_1(z,u),\sigma_2(z,u))$ with $\sigma_1(z,u)\in\mathbb{C}^N$, $\sigma_2(z,u)\in\U$.
Put $\overline{D'}:=\{(z,u)\in\mathbb{C}^N\times\U\mid (\overline{z},\overline{u})\in D'\}$.
Then for $(z,u,\overline{\zeta},\overline{v})\in\mathcal{D}'\times\overline{\mathcal{D}'}$, we have
\begin{equation}\label{r5YJcgyx}
\frac{1}{2\iu}(z-\overline{\zeta})-Q(u,v)
=\frac{1}{2\iu}(\sigma_1(\zeta,v)-\overline{\sigma_1(z,u)})-Q(\sigma_2(\zeta,v),\sigma_2(z,u)).
\end{equation}

Expand $\overline{\sigma_1(z,u)}$, $\sigma_1(\zeta,v)$, $\overline{\sigma_2(z,u)}$, and $\sigma_2(\zeta,v)$ into Taylor series around any point $(z_0,u_0,\overline{\zeta_0},\overline{v_0})\in\mathcal{D}'\times\overline{\mathcal{D}'}$ with respect to all the variables $z, u, \overline{\zeta}$, and $\overline{v}$ and compare the coefficients of monomials of the form
\begin{equation*}
p(z-z_0,u-u_0)q(\overline{\zeta}-\overline{\zeta_0},\overline{v}-\overline{v_0})\quad(p \text{ and }q \text{ are monomials of degree }1)
\end{equation*}
on both sides of \eqref{r5YJcgyx}. 
Then we see that there exists an anti-linear map $\mathbb{C}^N\oplus\U\ni(z,u)\mapsto d\sigma_2(z,u)\in\U$ such that 
\begin{equation*}
Q(u,v)=Q(d\sigma_2(\zeta,v),d\sigma_2(z,u)).
\end{equation*}
Letting $u=v=0$ and $z=\zeta$, we see that $d\sigma_2(z,u)=d\sigma_2(0,u)$ by the $\Omega$-positivity of $Q$.
\end{proof}

We shall present the proof of Theorem \ref{th:real form}, but in a slightly different form with a superfluous condition.
Let $W\subset \U$ be a real subspace.
For $\xi\in \left(\V\right)^*$, let $\omega_\xi(\cdot,\cdot):\U\times\U\rightarrow\mathbb{R}$ be the skew-symmetric bilinear form defined by
\begin{equation*}
\omega_\xi(u,v):=\langle \xi,[u,v]\rangle\quad(u,v\in\U).
\end{equation*}
Put
\begin{equation*}
W^{\perp,\omega_\xi}:=\{u\in \U \mid \omega_\xi(u,w)=0 \text{ for all }w\in W\}.
\end{equation*}
Let
\begin{equation*}
\Omega^*:=\{\xi\in(\V)^*\mid \langle\xi, y\rangle>0\text{ for all }y\in\overline{\Omega}\backslash\{0\}\},
\end{equation*}
which is also a regular cone.
The precise statement of a different version of Theorem \ref{th:real form} is given as follows.
\begin{theorem}\label{th:real form2}
The action of $G^W$ on $\mathcal{D}(\Omega,Q)$ is strongly visible with respect to an involutive anti-holomorphic diffeomorphism if and only if $W$ contains a real form $W_0$ of $\U$ such that
\begin{equation*}
W^{\perp,\omega_\xi}\subset W_0\quad(\xi\in\Omega^*)\quad\text{and}\quad\displaystyle\mathop{\mathrm{Im}} Q(W_0,W_0)=0.
\end{equation*}
\end{theorem}

\begin{proof}
For the ``if'' part, the map
\begin{equation*}\sigma:\mathbb{C}^N\times\U\ni(x+\iu y,u)\mapsto(-x+\iu y,-\overline{u}^{W_0})\in\mathbb{C}^N\times\U\quad(x,y\in\V, u\in\U)
\end{equation*} defines an anti-holomorphic diffeomorphism of $\mathcal{D}(\Omega,Q)$, where $\overline{u}^{W_0}$ denotes the complex conjugate of $u$ with respect to the real form $W_0$.
We have
\begin{equation*}
n(2x+4Q(w',w),2w)(-x+\iu y,-w+jw')=(x+\iu y,w+jw')\quad(w,w'\in W_0),
\end{equation*}
which shows that $\sigma$ preserves each $G^{W_0}$-orbit in $\mathcal{D}(\Omega,Q)$.
Also, we can confirm other conditions in Definition \ref{b5SgLeVG} (3).

For the ``only if'' part, we have
\begin{equation*}
W^{\perp,\omega_\xi}\subset W\quad(\xi\in\Omega^*),
\end{equation*}
which follows from the multiplicity-free theorem (see for instance \cite[Theorem 3]{FT99}, \cite[Corollary 2.3]{Kobayashi13}) together with Propositions \ref{Thecorresp} and \ref{z8KgCjZK} below.
Also, we have $W+jW=\U$, since the action of $G^W$ is visible (see \cite[Theorem 4]{Kobayashi05}).
Let $\xi\in\Omega^*$ and $P:=W\cap jW$.
Then we have
\begin{equation}
\U=W\oplus jW^{\perp,\omega_\xi},\quad W=W^{\perp,\omega_\xi}\oplus P.
\end{equation}
For $u\in\U$, let us denote by $u^\xi$ the projection of $u$ on $W^{\perp,\omega_\xi}\oplus j W^{\perp,\omega_\xi}$.
In the proof of Proposition \ref{Xn9bCyEc}, we now get another relation 
\begin{equation*}
\sigma_2(z,u)^\xi-\overline{\sigma_2(\zeta,v)^\xi}=v^\xi-\overline{u}^\xi\quad((z,u,\overline{\zeta},\overline{v})\in D'\times\overline{D'}),
\end{equation*}
where the bar denotes the complex conjugation
\begin{equation*}
W^{\perp,\omega_\xi}\oplus jW^{\perp, \omega_\xi}\ni w+jw'\mapsto w-jw'\in W^{\perp,\omega_\xi}\oplus jW^{\perp,\omega_\xi},
\end{equation*}
and hence $P$ is $A$-stable and we have
\begin{equation*}
(Aw)^\xi=-\overline{w^\xi}\quad(w\in \U).
\end{equation*}
We may assume that $A$ is also involutive, and if we put
\begin{equation*}
P_0:=\{u\in P\mid Au=-u\},
\end{equation*}
then
\begin{equation*}
W_0:=\{u\in \U\mid Au=-u\}=W^{\perp,\omega_\xi}\oplus P_0
\end{equation*}
satisfies the desired properties.
Indeed, it follows from \eqref{DFiWga8x} that
\begin{equation*}\label{eq:real}
Q(w,w')=Q(Aw',Aw)=Q(w',w)\quad(w,w'\in W_0).
\end{equation*}
\end{proof}

\begin{remark}
For a certain class of domains, the multiplicity-freeness property of $\pi_0|_{G^W}$ is represented by
\begin{equation}\label{eq:mf}
\mathop{\mathrm{Im}}Q(W^{\perp,\omega_\xi},W^{\perp,\omega_\xi})=0\quad(\xi\in\Omega^*),
\end{equation}
which will be dealt with in another paper.
Note that \eqref{eq:mf} is always satisfied for $W=\U$ and Theorem \ref{p9UgSPw2} is consistent with this fact.
On the other hand, the condition in Theorem \ref{th:real form2} coincides with \eqref{eq:mf} when $\dim W=M$.
\end{remark}

\section{Classification of equivariant holomorphic line bundles}\label{Ee3FnBHw}

Our aim in this section is to classify the isomorphism classes of $G$-equivariant holomorphic line bundles over $\mathcal{D}(\Omega,Q)$ (see Theorem \ref{Z8CrYFczzHas} below).
Our proof is based on the Oka-Grauert principle, which shows that any holomorphic line bundles over $\mathcal{D}(\Omega,Q)$ is holomorphically trivial, and the classifying problem descends to the one for holomorphic multipliers.
\begin{definition}
A smooth function $m:G\times \mathcal{D}(\Omega,Q)\rightarrow\mathbb{C}^\times$ is called a {\it multiplier} if the following cocycle condition is satisfied:
\begin{equation*}
m(gg',(z,u))=m(g,g'(z, u))m(g',(z, u))\quad(g,g'\in G, (z,u)\in \mathcal{D}(\Omega,Q)),
\end{equation*}
Moreover, a multiplier $m$ is called a {\it holomorphic multiplier} if $m(g,(z,u))$ is holomorphic in $(z,u)\in \mathcal{D}(\Omega,Q)$.
\end{definition}

For a holomorphic multiplier $m:G\times\mathcal{D}(\Omega,Q)\rightarrow \mathbb{C}^\times$, we define a $G$-equivariant holomorphic line bundle over $\mathcal{D}(\Omega,Q)$. Let us consider the action of $G$ on the trivial bundle $\mathcal{D}(\Omega,Q)\times\mathbb{C}$ given by 
\begin{equation}\label{gzzetagzmg}\begin{split}
\mathcal{D}(\Omega,Q)\times\mathbb{C}\ni ((z,u),\zeta)\mapsto(g(z,u),m(g,(z,u))\zeta)\in\mathcal{D}(\Omega,Q)\times\mathbb{C}\quad(g\in G).
\end{split}\end{equation}
By Chen \cite{Chen73}, the domain $\mathcal{D}(\Omega,Q)$ is a Stein manifold.
By the Oka-Grauert principle, every $G$-equivariant holomorphic line bundle over $\mathcal{D}(\Omega,Q)$ is isomorphic to the one defined as in \eqref{gzzetagzmg}.
Put $h(u,v):=\sum_{n=1}^Mu_n\overline{v_n}\,(u,v\in\U)$.
We have the following theorem.
\begin{theorem}\label{Z8CrYFczzHas}
The $G$-equivariant bundles defined as in \eqref{gzzetagzmg} with the one-dimensional representations 
\begin{equation*}
m_c:G\ni n(x_0,u_0)\mapsto e^{h(c,u_0)}\in\mathbb{C}^\times\quad(c\in \U)
\end{equation*}
form a complete set of representatives of the isomorphism classes of $G$-equivariant holomorphic line bundles over $\mathcal{D}(\Omega,Q)$.
\end{theorem}

Fix $\reference\in\Omega$.
To prove Theorem \ref{Z8CrYFczzHas}, we first see that $m_c\,(c\in \U)$ define mutually nonisomorphic $G$-equivariant holomorphic line bundles.
Suppose that the representations $m_c, m_{c'}\,(c,c'\in \U)$ define isomorphic $G$-equivariant holomorphic line bundles.
According to the next lemma, we have
\begin{equation}\label{mult2gzfgz}
m_{c'}(g,(z,u))=f_{c,c'}(g(z,u))f_{c,c'}(z,u)^{-1}m_c(g,(z,u))\quad(g\in G)
\end{equation}
for some holomorphic function $f_{c,c'}:\mathcal{D}(\Omega,Q)\rightarrow \mathbb{C}^\times$.
The lemma below will be used also in the proof of Theorem \ref{Z8CrYFczzHas}.
\begin{lemma}[{\cite[Lemma 1]{Ishi11}}]\label{equivalentline}
The trivial line bundles over $\mathcal{D}(\Omega,Q)$ together wih $G$-actions  \eqref{gzzetagzmg} for two holomorphic mutlipliers $m$ and $m'$ are isomorphic as $G$-equivariant holomorphic line bundles if and only if there exists a holomorphic function $f:\mathcal{D}(\Omega,Q)\rightarrow \mathbb{C}^\times$ such that
\begin{equation}
m'(g,(z,u))=f(g(z,u))f(z,u)^{-1}m(g,(z,u))\quad (g\in G).
\end{equation}
\end{lemma}
For a $\mathbb{C}^N\times\U$-valued function $X(z,u)$ on $\mathcal{D}(\Omega,Q)$, we denote by $D_X$ the differential operator defined by 
\begin{equation*}
\mathop{D_X}f(z,u):=\left.\frac{d}{dt}\right|_{t=0}f((z,u)+tX(z,u)).
\end{equation*}
Since $f_{c,c'}$ is holomorphic, for any $u_0\in \U$, we have
\begin{equation*}
(D_{(0,u_0)}+i D_{(0,ju_0)}f_{c,c'})(\iu p,0)=0.
\end{equation*}
For a Lie group $H$, let $dR$ denote the representation of $\mathfrak{h}_\mathbb{C}$ on $C^\infty(H)$ defined by
\begin{equation*}
dR(a_1+\iu a_2)f_0(g):=\left.\frac{d}{dt}\right|_{t=0} f_0(ge^{ta_1})+\iu\left.\frac{d}{dt}\right|_{t=0} f_0(ge^{ta_2})\quad(f_0\in C^\infty(H), a_1, a_2\in\mathfrak{h}). 
\end{equation*}
Then it follows from \eqref{mult2gzfgz} that
\begin{equation*}\begin{split}
2&h(c',u_0)=dR(u_0+\iu ju_0)m_{c'}(n(0,0))
\\&=(D_{(0,u_0)}+D_{(0,\iu ju_0)})f_{c,c'}(\iu \reference, 0)f_{c,c'}(\iu \reference,0)^{-1}+dR(u_0+\iu ju_0)m_c(n(0,0))
\\&=dR(u_0+\iu ju_0)m_c(n(0,0))=2h(c,u_0),
\end{split}\end{equation*}
which implies that $c=c'$.

\begin{proof}[Proof of \textup{Theorem \ref{Z8CrYFczzHas}}]
We have
\begin{equation}\begin{split}\label{QEie2dCKcfUd}
m(n&(x_0,u_0),(z,u))
\\&=m\left(n(x_0+2\mathop{\mathrm{Im}} Q(u_0,u),u_0+u\right), (z-\iu Q(u,u),0))
\\&\indent\cdot m(n(0,u),(z-\iu Q(u,u),0))^{-1}
\\&=m(n(0,u_0+u),(z-\iu Q(u,u)+x_0+2\mathop{\mathrm{Im}}Q(u_0,u),0))
\\&\indent\cdot m(n(x_0+2\mathop{\mathrm{Im}}Q(u_0,u),0),(z-\iu Q(u,u),0))
\\&\indent \cdot m(n(0,u),(z-\iu Q(u,u),0))^{-1}
\\&=m'(n(x_0,u_0)(z,u))m'(z,u)^{-1},
\end{split}\end{equation}
where we put
\begin{equation*}
m'(z,u):=m(n(0,u),(z-\iu Q(u,u),0))m(n(x,0),iy-iQ(u,u))\quad(z=x+iy,x,y\in\V).
\end{equation*}

Let $f_m(z,u):=\log m'(z,u)$ be defined on a neighborhood of a point $(x'+iy',u')\in\mathcal{D}(\Omega,Q)$.
In what follows, if necessary, we may take a smaller neighborhood as the domain of definition of $f_m$.
We have
\begin{equation}\label{x8PBtaZmQ2Zt}
f_m(n(x_0,u_0)(z,u))-f_m(z,u)=\log m(n(x_0,u_0),(z,u)),
\end{equation}
and hence
\begin{equation}
\frac{\partial}{\partial \overline{z_1}}f_m(n(x_0,u_0)(z,u))=\frac{\partial}{\partial \overline{z_1}}f_m(z,u).
\end{equation}
There exists $G_1\in C^\infty(\Omega)$ such that $\frac{\partial}{\partial \overline{z_1}}f_m(z,u)=G_1(y-Q(u,u))$, and if we put \begin{equation*}
F_1(y):=-2\iu \int_{(y'-Q(u',u'))_1}^{y_1} G_1(t,y_2,y_3,\cdots,y_{N})\,dt,
\end{equation*}
then it follows that there exists a locally defined smooth function $K_1$ holomorphic in $z_1$ such that
\begin{equation*}
f_m(z,u)=F_1(y-Q(u,u))+K_1(z,u).
\end{equation*}
Then we get inductively that for $k=1,2,\cdots,N$, there exist $F_k\in C^\infty(\Omega)$ and a locally defined smooth function $K_k$ holomorphic in $z_1,z_2,\cdots ,z_k$ such that
\begin{equation*}
f_m(z,u)=F_k(y-Q(u,u))+K_k(z,u).
\end{equation*}
Put $F:=F_N$ and $K:=K_N$.

Now \eqref{x8PBtaZmQ2Zt} reads 
\begin{equation*}
K(n(x_0,u_0)(z,u))-K(z,u)=\log m(n(x_0,u_0),(z,u))\end{equation*}
and since $K$ is holomorphic in $z$, we have
\begin{equation*}
\frac{\partial}{\partial \overline{u_\alpha}}K(n(x_0,u_0)(z,u))=\frac{\partial}{\partial \overline{u_\alpha}}K(z,u)\quad(\alpha=1,2,\cdots, M).
\end{equation*}
Thus for $\alpha=1,2,\cdots, M$, there exists $H_\alpha\in C^\infty(\Omega)$ such that $\frac{\partial}{\partial \overline{u_\alpha}}K(z,u)=H_\alpha(y-Q(u,u))$, and hence $\frac{\partial}{\partial \overline{u_\alpha}}K(z,u)=c_\alpha\in \mathbb{C}$, as $\frac{\partial}{\partial \overline{u_\alpha}}K(z,u)$ is holomorphic in $z$.
It follows that $K(z,u)=h(c,u)+L(z,u)$ with $L$ holomorphic, $c\in\U$.
Hence
\begin{equation}
f_m(z,u)=F(y-Q(u,u))+h(c,u)+L(z,u).
\end{equation}
Since $L$ is holomorphic and $\mathcal{D}(\Omega,Q)$ has a trivial homology, the locally defined functions $L$ and $F$ extend to whole $\mathcal{D}(\Omega,Q)$ in such a way that
\begin{equation*}
e^{F(y-Q(u,u))+L(z,u)+h(c,u)}=m'(z,u).
\end{equation*}
Consequently, by \eqref{QEie2dCKcfUd}, we have 
\begin{equation*}
m(n(x_0,u_0),(z,u))=e^{h(c,u_0)+L(n(x_0,u_0)(z,u))-L(z,u)}.
\end{equation*}
By Lemma \ref{equivalentline}, our assertion follows.
\end{proof}

\section{Representations of the generalized Heisenberg group}\label{u6EAbPUB}
In this section we show Theorem \ref{p9UgSPw2} and classify all irreducible unitary representations of $G$ realized in $\mathcal{O}(\mathcal{D}(\Omega,Q))$.
According to the classification in the previous section, we can focus ourselves on representations having quite simple descriptions.

\subsection{Restriction to the center $Z(G)$}

For a complex manifold $D$, we denote by $\mathcal{O}(D)$ the space of holomorphic functions on $D$.
Let $c\in \U$ and $\pi_c$ a linear representation of $G$ on $\mathcal{O}(\mathcal{D}(\Omega,Q))$ given by
\begin{equation*}
\pi_c(g)f(z,u):=e^{h(c,u_0)}f(g^{-1}(z,u))\quad(f\in\mathcal{O}(\mathcal{D}(\Omega,Q)), g=n(x_0,u_0)\in G).
\end{equation*}
\begin{definition}[cf. \cite{FT99}]\label{def:invariant hilbert}
For a unitary representation $(\pi,\mathcal{H})$ of $G^W$, we say $\pi$ is {\it realized in} $\mathcal{O}(\mathcal{D}(\Omega,Q))$, or $\mathcal{H}$ is a $G^W${\it-invariant Hilbert subspace} of $\mathcal{O}(\mathcal{D}(\Omega,Q))$, given an injective continuous $G^W$-intertwining operator between $\pi$ and $\pi_c$ with respect to the compact-open topology of $\mathcal{O}(\mathcal{D}(\Omega,Q))$.
In this case, we use the terminology `{\it irreducible $G^W$-invariant Hilbert subspace}' when $\pi$ is irreducible.
\end{definition}
We consider the following condition for $\pi_c|_{G^W}$.
\begin{definition}[cf. {\cite{FT99}}, {\cite{Thomas}}]\label{def:MFC}
We say $\pi_c|_{G^W}$ is {\it multiplicity-free} if any two irreducible $G^W$-invariant Hilbert subspaces of $\mathcal{O}(\mathcal{D}(\Omega,Q))$ either coincide as linear spaces, and have proportional inner products, or they yield inequivalent representations of $G^W$.
\end{definition}
\begin{remark}\label{rem:mfc}
According to \cite[Theorem 2]{FT99}, when $c=0$, the notion of multiplicity-freeness in Definition \ref{def:MFC} is equivalent to the following:
\begin{equation*}
\text{Any unitary representation of $G^W$ realized in }\mathcal{O}(\mathcal{D}(\Omega,Q))\text{ is multiplicity-free.}
\end{equation*}
For the definition of multiplicity-free unitary representation, see Definition \ref{def:mfu} below.
\end{remark}
For $\xi\in\mathfrak{g}^*$, we denote by $\pi^\xi$ an irreducible unitary representation of $G$ corresponding to the coadjoint orbit through $\xi$ by the Kirillov-Bernat correspondence.
Suppose that an irreducible unitary representation $(\pi,\mathcal{H})$ of $G$ is realized in $\mathcal{O}(\mathcal{D}(\Omega,Q))$ with $\Phi_0:\mathcal{H}\hookrightarrow\mathcal{O}(\mathcal{D}(\Omega,Q))$ and is equivalent to $\pi^\nu$ with $\nu\in\mathfrak{g}^*$.
Let $K^\mathcal{H}$ be the reproducing kernel defined by $\Phi_0$, and put $K^{\mathcal{H}}_{(\iu p,0)}(z,u):=K^\mathcal{H}(z,u,\iu p,0)$.
For $a=a_1+\iu a_2\in\mathfrak{g}_\mathbb{C}$ with $a_1, a_2\in\mathfrak{g}$ and $\xi\in\mathfrak{g}^*$, we write $\langle\xi, a\rangle=\langle \xi, a_1\rangle+\iu \langle\xi, a_2\rangle$.
Let us denote by $\mathop{\mathrm{Ad}^*}$ the coadjoint representation of $G$.

\begin{proposition}\label{Ksqrtemath}
We have $K_{(\iu \reference,0)}^{\mathcal{H}}(z,u)=e^{-\iu \langle\nu, z\rangle}F(u)$ for some $F\in\mathcal{O}(\U)$.
\end{proposition}
\begin{proof}
One has
\begin{equation}\label{leftmathop}
\left(\mathop{\mathrm{Ad}^*(g)}\xi\right)|_{\mathfrak{z(g)}}=\xi|_{\mathfrak{z(g)}}\quad(g\in G).
\end{equation}
From \eqref{leftmathop} and the description of the decomposition of an irreducible unitary representation of a nilpotent Lie group \cite{CG88}, we can write the disintegration of $\pi^\nu|_{Z(G)}$ as
\begin{equation*}
\pi^\nu|_{Z(G)}\simeq\int_{\widehat{Z(G)}}n(\sigma)\sigma\,d\mu(\sigma)\simeq n(\sigma^{\nu})\sigma^{\nu},
\end{equation*}
with $\mu$ a Borel measure on the unitary dual $\widehat{Z(G)}$, $n(\sigma)=0,1,2,\cdots,\infty$, and $\sigma^{\nu}$ a unitary representation of $Z(G)$ corresponding to the orbit $\{\nu|_{\mathfrak{z(g)}}\}\subset\mathfrak{z(g)}^*$. 
On the other hand, the commutativity of $Z(G)$ implies that every irreducible $Z(G)$-invariant Hilbert subspace of $\mathcal{O}(\mathcal{D}(\Omega,Q))$ is one-dimensional and given by the linear span of a function $e^{-\iu \langle \eta, z\rangle}F(u)$ with $\eta\in \left(\V\right)^*$ and $F\in\mathcal{O}(\U)$.
Note that the $Z(G)$-module corresponds to the orbit $\{\eta\}\subset\mathfrak{z(g)}^*$. 
Therefore we can take an orthonormal basis $\{e_k\}_{k=1}^{n(\sigma^\nu)}$ of $\mathcal{H}$ such that $e_k(z,u)=e^{-\iu \langle\nu,z\rangle}F_k(u)$ with $F_k\in\mathcal{O}(\U)$.
Then we have for $(w,v)\in\mathcal{D}(\Omega,Q)$
\begin{equation*}
K^\mathcal{H}(z,u,w,v)=\sum_{k=1}^{n(\sigma^\nu)}e_k(z,u)\overline{e_k(w,v)}=e^{\iu \langle\nu,z-\overline{w}\rangle}\sum_{k=1}^{n(\sigma^\nu)}F_k(u)\overline{F_k(v)},
\end{equation*}
which implies the assertion.
\end{proof}

\subsection{Coherent state representation}
For a unitary representation $(\tau,\mathcal{W})$ of $G$, let us regard the projective space $\mathbb{P}(\mathcal{W})$ as a (possibly infinite-dimensional) K\ua hler manifold and consider the action of $G$ given by $\mathbb{P}(\mathcal{W})\ni[v]\mapsto [\tau(g)v]\in\mathbb{P}(\mathcal{W})\,(g\in G)$.
\begin{definition}[{\cite[Definition 4.2]{Lisiecki95}}]
We call $\tau$ a {\it coherent state representation} (CS representation for short) of $G$ if there exists a $G$-orbit in $\mathbb{P}(\mathcal{W})$ which is a complex submanifold of $\mathbb{P}(\mathcal{W})$ and does not reduce to a point.
\end{definition}
The reproducing kernel $K^\mathcal{H}$ is $G$-invariant, and hence one has
\begin{equation}\label{dpioverlin}
\mathop{d\pi_c(a)} K^\mathcal{H}_{(\iu \reference,0)}=-2h(u,c)K_{(\iu \reference,0)}^\mathcal{H}
\end{equation}
for $a=u-\iu ju$ with $u\in \U$, where we abbreviate $\pi_c|_\mathcal{H}$ to $\pi_c$ and extend the differential representation $d\pi_c$ to a representation of $\mathfrak{g}_\mathbb{C}$ by the complex linearity.
Let $\mathfrak{h}_-$ be the complex subalgebra of $\mathfrak{g}_\mathbb{C}$ given by
\begin{equation*}
\mathfrak{h}_-:=\mathbb{C}^N\oplus\{u+\iu ju\mid u\in{\U}\}\subset\mathfrak{g}_\mathbb{C}.
\end{equation*}
Put
\begin{equation*}
\tilde{\nu}(x,u):=\langle\nu, x\rangle+2\mathop{\mathrm{Im}}h(c,u)\quad(x\in \V, u\in \U).
\end{equation*}
By \eqref{dpioverlin} and Proposition \ref{Ksqrtemath}, we see that $K_{(\iu \reference,0)}^\mathcal{H}$ solves the following equation:
\begin{equation}\label{dpiKmathca}
\mathop{d \pi_c(a)}f=\iu \langle\tilde{\nu},a\rangle f\quad(a\in\overline{\mathfrak{h}_-},f\in\mathcal{H}),
\end{equation}
where the bar denotes the complex conjugation of $\mathfrak{g}_\mathbb{C}$ with respect to the real form $\mathfrak{g}$.
By \cite[2., Proposition]{Lisiecki90}, this implies that $\pi$ is a coherent state representation of $G$, and the linear form $\xi=-\tilde{\nu}$ satisfies
\begin{equation}\label{tildenuove}
-\iu \langle \xi,[a,\overline{a}]\rangle=8\langle\xi,Q(u,u)\rangle\geq 0
\end{equation}
for all $a=u+\iu ju$ with $u\in\U$.
Here we note that the moment map $\mu_\pi:\mathbb{P}(\mathcal{H}^\infty)\rightarrow\mathfrak{g}^*$ defined by
\begin{equation*}
\langle \mu_{\pi}([\psi]),x\rangle=\frac{1}{i}\frac{(d\pi(x)\psi,\psi)_\mathcal{H}}{(\psi,\psi)_\mathcal{H}}\quad(x\in\mathfrak{g},\psi\in\mathcal{H}^\infty)
\end{equation*}
satisfies
\begin{equation*}
\langle \mu_{\pi}([K_{(ip,0)}^\mathcal{H}]),a\rangle=\langle\widetilde{\nu},a\rangle
\end{equation*}
for all $a\in\overline{\mathfrak{h}_-}$.
Hence we have $\mu_\pi([K_{(ip,0)}^\mathcal{H}])=\widetilde{\nu}$.
We have the following theorem.
\begin{theorem}\label{ZHwubsUKw23T}
Any infinite dimensional irreducible unitary representation of $G$ realized in $\mathcal{O}(\mathcal{D}(\Omega,Q))$ is a CS representation.
\end{theorem}

\subsection{Multiplicity-freeness}
Now we show that $f(z,u)=e^{-\iu \langle \nu,z\rangle}e^{h(u,c)}$ is only the solution of \eqref{dpiKmathca} up to a constant.
First, for $a=u_0-\iu ju_0$ with $u_0\in{\U}$, the equation \eqref{dpiKmathca} tells us that
\begin{equation*}\begin{split}
\mathop{\left(D_{(-2\iu Q(u,u_0),-u_0)}-\iu D_{(-2\iu Q(u,ju_0),-ju_0)}\right)}f(z,u)=-2h(u_0,c)f(z,u).
\end{split}\end{equation*}
Since $f$ is holomorphic in $z$ and $Q$ is Hermitian, it follows that
\begin{equation*}
\mathop{\left(D_{(0,u_0)}-\iu D_{(0,ju_0)}\right)}f(z,u)=2h(u_0,c)f(z,u),
\end{equation*}
which implies that there exists $F\in\mathcal{O}(\mathcal{D}(\Omega))$ such that 
\begin{equation}\label{fzuFzzuinm}
f(z,u)=F(z)e^{h(u,c)}.
\end{equation}
Next, for $a=x_0\in \V$, the equation \eqref{dpiKmathca} can be read as
\begin{equation*}
\mathop{D_{(-x_0,0)}}f(z,u)=\iu \langle \nu, x_0\rangle f(z,u),
\end{equation*}
from which we can see that 
\begin{equation}\label{fzxuesqrtl}
f(z+x_0,u)=e^{-\iu \langle \nu,x_0\rangle}f(z,u).
\end{equation}
Combining \eqref{fzuFzzuinm} with \eqref{fzxuesqrtl}, we have
\begin{equation*}\begin{split}
f(z,u)&=F(z)e^{h(u,c)}=F(\iu \reference+z-\iu \reference)e^{h(u,c)}
\\&=F(\iu \reference)e^{-\iu \langle \nu,z-\iu \reference\rangle}e^{h(u,c)},
\end{split}\end{equation*}
by the analytic continuation.
Thus $f(z,u)=e^{-\iu \langle \nu,z\rangle}e^{h(u,c)}$ is only the solution of \eqref{dpiKmathca} up to a constant.
We obtain the following theorem.
\begin{theorem}\label{tyjqndhg5v55}
If an irreducible unitary representation $(\pi,\mathcal{H})$ of $G$ corresponding to $\nu\in\mathfrak{g}^*$ is realized in $\mathcal{O}(\mathcal{D}(\Omega,Q))$, then 
the reproducing kernel satisfies
\begin{equation*}
K_{(ip,0)}^\mathcal{H}(z,u)=e^{-i\langle \nu,z\rangle}e^{h(u,c)}
\end{equation*}
up to a constant.
\end{theorem}
As a corollary of Theorem \ref{tyjqndhg5v55}, we now show Theorem \ref{p9UgSPw2}.
\begin{proof}[Proof of \textup{Theorem \ref{p9UgSPw2}}]
Let $\Phi:\mathcal{H}\hookrightarrow \mathcal{O}(\mathcal{D}(\Omega,Q))$ be an arbitrary injective continuous $G$-intertwining operator between $\pi$ and $\pi_c$.
Then $\Phi\Phi_0^{-1} K_{(\iu p,0)}^\mathcal{H}$ solves \eqref{dpiKmathca}, and hence coincides with $K_{(\iu p,0)}^\mathcal{H}$ up to a scalar multiplication by Theorem \ref{tyjqndhg5v55}.
By the $G$-invariance and the reproducing property of the kernel function, we see that any two equivalent irreducible $G$-invariant Hilbert subspaces coincide as linear spaces and have proportional inner products.
This completes the proof.
\end{proof}

\subsection{Classification of irreducible invariant Hilbert subspaces}
Let $\xi\in\mathfrak{g}^*$ be satisfying $\xi|_{\U}=0$, which may be expressed as $\xi\in(\V)^*$, and \eqref{tildenuove}.
Put
\begin{equation*}
N_\xi:=\{u\in \U|\,\langle \xi,Q(u,u)\rangle=0\}\quad\text{and}\quad N_\xi^\perp:=\{u\in \U\mid h(u,N_\xi)=0\}.
\end{equation*}
For a finite dimensional real vector space $V$, we identify naturally $V$ with the Euclidean space of the same dimension, and let $d\lambda_V$ denote the pushforward measure of the Lebesgue measure.
Put 
\begin{equation*}
\mathcal{F}_\xi:=\left\{F\in\mathcal{O}({\U})\left|\,\begin{array}{c}F(u+v)=F(u)\makebox{\quad for all }u\in \U\makebox{ and }v\in N_\xi,\\ \int_{{N_\xi}^{\perp}}|F(u)|^2e^{-2\langle\xi, Q(u,u)\rangle}\,d\lm_\xi(u)<\infty\end{array}\right.\right\},
\end{equation*}
where $d\lm_\xi:=d\lambda_{N_\xi^\perp}$ is normalized so that $\int_{N_\xi^\perp}e^{-2\langle\xi,Q(u,u)\rangle}d\lm_\xi(u)=1$.
For $\cc\in\U$, let $V_{\xi,\cc}$ be the representation of $G$ on $\mathcal{F}_\xi$ defined by
\begin{eqnarray*}
\mathop{V_{\xi,\cc}(n(x_0,u_0))}F(u)
\\:=e^{-\iu \langle \xi,x_0\rangle}e^{2\iu \mathop{\mathrm{Im}}h(\cc,u_0)}e^{-\langle \xi,Q(u_0,u_0)\rangle}e^{2\langle \xi,Q(u,u_0)\rangle}F(u-u_0)
\\(x_0\in{\V},u_0\in{\U}).
\end{eqnarray*}
\begin{proposition}\label{Thecorresp}
The operator $\Phi_\xi$ intertwines $V_{\xi,c}$ with $\pi_c$.
\end{proposition}
\begin{proof}
We have
\begin{equation*}\begin{split}
\pi_c(\exp x_0)\mathop{\Phi_\xi} F(z,u)&=\mathop{\Phi_\xi} F(z-x_0,u)=e^{h(u,c)}e^{\langle\xi,\iu (z-x_0)\rangle}F(u)
\\&=e^{h(u,c)}e^{\langle \xi, \iu z\rangle}\mathop{V_{\xi,c}(\exp x_0)}F(u)=\mathop{\Phi_\xi} \mathop{V_{\xi,c}(\exp x_0)}F(z,u).
\end{split}\end{equation*}
Also, we have
\begin{equation*}\begin{split}
\mathop{\pi_c(\exp u_0)}&=e^{h(u-u_0,c)}e^{h(c,u_0)}e^{\langle\xi,\iu z+2Q(u,u_0)-Q(u_0,u_0)\rangle}F(u-u_0)
\\&=e^{h(u,c)}e^{\langle\xi,\iu z\rangle}\mathop{V_{\xi,c}(\exp u_0)}F(u)=\mathop{\Phi_\xi} \mathop{V_{\xi,c}(\exp u_0)}F(z,u).
\end{split}\end{equation*}
This completes the proof.
\end{proof}

Let ${X_{\xi,s}}\in\mathfrak{g}^*$ be given by
\begin{equation*}
{X_{\xi,s}}(x,u):=-\langle\xi,x\rangle+2\mathop{\mathrm{Im}}h(s,u).
\end{equation*}
For $u\in\U$, let $u_\xi\in N_\xi$, and $u^\xi\in N_\xi^\perp$ be the orthogonal projections of $u$ on $N_\xi$ and $N_\xi^\perp$, respectively.
\begin{lemma}\label{lem:induction}
The image of the coadjoint orbit $\Ad^*(G)X_{\xi,s}\subset\mathfrak{g}^*$ under the Kirillov-Bernat map is the unitary equivalence class of $V_{\xi,s_\xi}$.
\end{lemma}
\begin{proof}
First we reconstruct $V_{\xi,s_\xi}$ via the Auslander-Kostant theory.
Let
\begin{equation*}
\mathfrak{p}:=\mathbb{C}^N\oplus\left(N_\xi\right)_\mathbb{C}\oplus\{u+\iu ju\mid u\in N_\xi^\perp\},\quad\mathfrak{d}:=\mathfrak{p}\cap \mathfrak{g}, \quad\text{and}\quad D:=Z(G)\exp(N_\xi).
\end{equation*}
Then $\mathfrak{p}$ is a positive polarization of $\mathfrak{g}$ at ${X_{\xi,s}}\in\mathfrak{g}^*$ satisfying the Pukanszky condition.
Let $\mathcal{H}({X_{\xi,s}},\mathfrak{p},G)$ be the space of of smooth functions $\phi$ on $G$ satisfying
\begin{eqnarray}
\phi(g\exp b)=e^{-\iu\langle{X_{\xi,s}}, b\rangle}\phi(g)\quad(g\in G, b\in \mathfrak{d}),
\\
\int_{G/D} |\phi|^2\,d\mu_{G/D}<\infty,
\end{eqnarray}
and
\begin{eqnarray}
\mathop{dR(q)}\phi=-\iu \langle {X_{\xi,s}}, q\rangle \phi\quad(q\in\mathfrak{p}),
\end{eqnarray}
where $\mu_{G/D}$ denotes a nonzero $G$-invariant measure on $G/D$.
The holomorphic induced representation $\rho=\rho({X_{\xi,s}},\mathfrak{p},G)$ is defined by
\begin{equation*}
\mathop{\rho(g)}\phi(g_0):=\phi(g^{-1}g_0)\quad(\phi\in \mathcal{H}({X_{\xi,s}},\mathfrak{p},G), g,g_0\in G).
\end{equation*}
Let $\Psi_{\xi,s}:\mathcal{H}({X_{\xi,s}},\mathfrak{p},G)\rightarrow\mathcal{O}(\U)$ be given by
\begin{equation*}
\mathop{\Psi_{\xi,s}} \phi(u)=e^{\langle\xi,Q(u^\xi,u^\xi)\rangle}\phi(n(0,u^\xi))\quad(u\in \U).
\end{equation*}
Then $\Psi_{\xi,s}$ gives a $G$-intertwining operator from $\mathcal{H}({X_{\xi,s}},\mathfrak{p},G)$ onto $\mathcal{F}_{\xi}$.
The unitary equivalence class of the representation $\rho({X_{\xi,s}}, \mathfrak{p}, G)$ is mapped by the Kirillov-Bernat map into the orbit $G {X_{\xi,s}}$ by Fujiwara \cite{Fujiwara74}, and hence the assertion follows.
\end{proof}
The reproducing kernel $K^\mathcal{H}$, which is partially determined in Theorem \ref{tyjqndhg5v55}, coincides with the reproducing kernel defined by $\Phi_\xi$ with $\xi|_{\mathfrak{z}(\mathfrak{g})}=-\nu|_{\mathfrak{z}(\mathfrak{g})}$, up to a constant.
Thus if we put
\begin{equation*}
P:=\{\xi\in\mathfrak{g}^*\mid\xi\text{ satisfies }\xi|_{\U}=0\text{ and }\eqref{tildenuove}\},
\end{equation*}
then we have the following.
\begin{corollary}
$\{V_{\xi,c}\}_{\xi\in P}$ is a set of complete representatives of the equivalence classes of irreducible unitary representations of $G$ realized in $\mathcal{O}(\mathcal{D}(\Omega,Q))$.
\end{corollary}
We shall complete the classification of $V_{\xi,s}\,(\xi\in P,s\in\U)$.
\begin{proposition}\label{W2deVQNG}
For $s, t\in \U$, the representations $V_{\xi,s}$ and $V_{\xi,t}$ of $G$ are unitarily equivalent if and only if $s-t\in N_\xi^{\perp}$.
When this condition is satisfied, the following operator $\psi_{\cc,t}$ intertwines $(V_{\xi, \cc},\mathcal{F}_\xi)$ with $(V_{\xi,t},\mathcal{F}_\xi)$\textup{:}
\begin{equation*}
\mathop{\psi_{s,t}} F(u):=\int_{N_\xi^\perp} F(v)e^{2\iu \mathop{\mathrm{Im}}h(t-\cc,v)}e^{2\langle \xi, Q(u,v)\rangle}e^{-2\langle \xi, Q(v,v)\rangle}\,d\lm_\xi(v).
\end{equation*}
\end{proposition}
\begin{proof}
For $u,w\in \U$ and $F\in \mathcal{F}_\xi$, we have
\begin{equation*}\begin{split}
\mathop{\psi_{\cc,t}}&\mathop{V_{\xi,\cc}(n(0,w))}F(u)
\\&=e^{2\iu \mathop{\mathrm{Im}}h(\cc,w)}e^{-\langle \xi, Q(w,w)\rangle}
\\&\quad\cdot\int_{N_\xi^\perp}F(v-w)e^{2\iu \mathop{\mathrm{Im}}h(t-\cc,v)+2\langle \xi,Q(v,w)\rangle+2\langle\xi,Q(u,v)\rangle-2\langle\xi,Q(v,v)\rangle}\,d\lm_\xi(v)
\\&=e^{2\iu \mathop{\mathrm{Im}}h(t-s,w^\xi)+2\iu \mathop{\mathrm{Im}}h(s,w)+2\langle\xi,Q(u,w)\rangle-\langle\xi,Q(w,w)\rangle}
\\&\quad\cdot \int_{N_\xi^\perp}F(v)e^{2\iu \mathop{\mathrm{Im}}h(t-\cc,v)+2\langle\xi, Q(u-w,v)\rangle-2\langle\xi,Q(v,v)\rangle}\,d\lm_\xi(v)
\\&=e^{2\iu \mathop{\mathrm{Im}}h(s-t,{w}_\xi)}\mathop{V_{\xi,t}(n(0,w))}\mathop{\psi_{\cc,t}} F(u).
\end{split}
\end{equation*}
This shows the ``if'' part of the statement.
Conversely, suppose that $V_{\xi,s}\simeq V_{\xi,t}$.
By the ``if'' part of the statement and Lemma \ref{lem:induction}, we see that the unitary equivalence class of $V_{\xi,s}$ corresponds to $\Ad^*(G)X_{\xi,s_\xi}\subset \mathfrak{g}^*$.
From the equality $\Ad^*(G)X_{\xi,s_\xi}=\Ad^*(G)X_{\xi,t_\xi}$ we see that there exists $u\in\U$ such that
\begin{equation*}
\mathop{\mathrm{Im}} h((s-t)_\xi, v)=\langle \xi, [u,v]\rangle\quad(v\in\U).
\end{equation*}
Putting $v=j(s-t)_\xi$ in the above equality, we see that $s-t\in N_\xi^\perp$.
\end{proof}

\section{Subgroups of the generalized Heisenberg group}\label{Fi4dG3gb}
In this section we study several properties that ensure the multiplicity-freeness of $\pi_0|_{G^W}$, especially in the case that $\dim W=M$ as in Theorem \ref{X6mQfvEy}.

\subsection{Multiplicity-free unitary representation}
\begin{definition}\label{def:mfu}
For a unitary representation $(\tau,\mathcal{W})$ of $G^W$, we say $\tau$ is {\it multiplicity-free} if the ring $\mathrm{End}_{G^W}(\mathcal{W})$ of continuous endomorphism commuting with $G^W$ is commutative.
\end{definition}
\begin{remark}\label{rem:mfu}
The notion of multiplicity-freeness in Definition \ref{def:mfu} can be rephrased in terms of the multiplicity function determined by the direct integral decomposition (see {\cite[Proposition 1.5.1]{Kobayashi05}}).
\end{remark}

\begin{proposition}\label{z8KgCjZK}
Let $\xi\in P$.
The unitary representation $V_{\xi,c}$ is multiplicity-free when restricted to the subgroup $G^W \subset G$ if and only if for any $u\in W^{\perp, \omega_\xi}$ there exists $w\in W$ such that $u+w\in N_\xi$.
\end{proposition}

\begin{proof}
Let $\mathrm{proj}:\mathfrak{g}^*\rightarrow \left({\mathfrak{g}^W}\right)^*$ be the natural projection.
We can write disintegration of $V_{\xi,c}|_{G^W}$ as
\begin{equation*}
V_{\xi,c}|_{G^W}\simeq\int_{\widehat{G^W}}n(\sigma)\sigma\,d\mu(\sigma)
\end{equation*}
with $\mu$ a Borel measure on $\widehat{G^W}$ and $n(\sigma)$ the number of $G^W$-orbits in $\mathrm{proj}^{-1}(\Ad^*(G^W)(\eta+\eta'))\cap \Ad^*(G)(-\xi)$
when $\sigma$ corresponds to
\begin{equation*}
\Ad^*(G^W)(\eta+\eta')\in \left(\V\right)^*\oplus W^*\quad(\eta\in(\V)^*,\eta'\in W^*)
\end{equation*}
by the Kirillov-Bernat map.
For $u\in\U$, let $\xi_u\subset\left(\U\right)^*$ be given by
\begin{equation*}
\langle \xi_u, u'\rangle:=\langle\xi,[u,u']\rangle\quad(u'\in\U).
\end{equation*}
We have
\begin{equation*}
\Ad^*(G)X_{\xi,c}=\{X_{\xi,c}+\xi_u\in\left(\V\right)^*\oplus\left(\U\right)^*\mid u\in\U\}
\end{equation*}
and 
\begin{equation*}
\Ad^*(G^W)(\eta+\eta')=\{\eta+\left(\eta'+\eta_w|_W\right)\in\left(\V\right)^*\oplus W^*\mid w\in W\}
\end{equation*}
Suppose $\mathrm{proj}^{-1}\left(\Ad^*(G^W)(\eta+\eta')\right)\cap \Ad^*(G)X_{\xi,c}\neq \emptyset$ and take an element $X_{\xi,c}+\xi_{u_0}$ with $u_0\in \U$.
Then we have $-\xi=\eta$ and
\begin{equation*}
\mathrm{proj}^{-1}\left(\Ad^*(G^W)(\eta+\eta')\right)\cap \Ad^*(G)X_{\xi,c}=\{X_{\xi,c}+\xi_u\mid u\in u_0+W^{\perp,\omega_\xi}+W\}.
\end{equation*}
Since the above set contains the coadjoint orbit
\begin{equation*}
\Ad^*(G^W)(X_{\xi,c}+\xi_{u_0})=\{X_{\xi,c}+\xi_{u}\mid u\in u_0+W\},
\end{equation*} 
the number of $G^W$-orbits in the set equals $1$ if and only if for any $u\in W^{\perp, \omega_\xi}$ there exists $w\in W$ such that $u+w\in N_\xi$.
Moreover, according to Remark \ref{rem:mfu}, the condition holds if and only if $\mathrm{End}_{G^W}(V_{\xi,c})$ is commutative.
\end{proof}
\subsection{Coisotropic action and invariant differential operators}
We see an integral expression for the Bergman kernel of $\mathcal{D}(\Omega,Q)$.
Let $\lambda$ denote the Lebesgue measure or the measure defined on the dual space via the standard inner product.
For $\xi\in \Omega^*$, let
\begin{equation*}
I(\xi):=\int_{\Omega}e^{-2\langle \xi, y\rangle}\, d\lm(y)\quad\text{ and }\quad I_Q(\xi):=\int_{\U} e^{-2\langle\xi,Q(u,u)\rangle}\, d\lm(u).
\end{equation*}
Let $L_a^2(\mathcal{D}(\Omega,Q)):=L^2(\mathcal{D}(\Omega,Q), \lm)\cap \mathcal{O}(\mathcal{D}(\Omega,Q))$.

\begin{theorem}[\cite{Gindikin}]\label{K3tmLfdV}
The reproducing kernel $K$ of $L_a^2(\mathcal{D}(\Omega,Q))$ is given by 
\begin{eqnarray*}
K(z,u,w,v):=\frac{1}{(2\pi)^{N}}\int_{\Omega^*}e^{\iu \langle \xi, z-\overline{w}-2\iu Q(u,v)\rangle}I(\xi)^{-1}I_Q(\xi)^{-1}\,d\lm(\xi).
\end{eqnarray*}
\end{theorem}
Suppose $W\subset \U$ is a real form.
Fix bases $\{e_k\}_{1\leq k\leq N}$ and $\{f_{\alpha}\}_{1\leq \alpha\leq M}$ of $\V$ and $W$, respectively.
From now on, we write the coordinates of vectors in $\mathbb{C}^N$ and $\U$ over the bases by $z_1, z_2,\cdots, z_N$ and $u_1, u_2,\cdots, u_M$, respectively.
For $1\leq k,l\leq N$, and $1\leq \alpha,\beta\leq M$, the value $g_{(z,u)}$ at $(z,u)\in\mathcal{D}(\Omega,Q)$ of the Bergman metric on $\mathcal{D}(\Omega,Q)$ is expressed as
\begin{equation}\label{T2xreVk4}\begin{split}
g_{k\overline{l}}=\frac{\partial^2}{\partial z_k\partial \overline{z_l}}\log K(z,u,z,u)=-(2\pi)^{-2N}K^{-2}(z,u,z,u)\left\{(2\pi)^{N}K(z,u,z,u)\right.
\\\cdot\int_{\Omega^*}\langle\xi,e_k\rangle \langle\xi, e_l\rangle e^{i\langle\xi,z-\overline{z}-2\iu Q(u,u)\rangle}I(\xi)^{-1}I_Q(\xi)^{-1}\,d\lm(\xi)
\\-\int_{\Omega^*}\langle\xi,e_k\rangle e^{i\langle\xi,z-\overline{z}-2\iu Q(u,u)\rangle}I(\xi)^{-1}I_Q(\xi)^{-1}\,d\lm(\xi)
\\\left.\cdot\int_{\Omega^*}\langle\xi, e_l\rangle e^{i\langle\xi,z-\overline{z}-2\iu Q(u,u)\rangle}I(\xi)^{-1}I_Q(\xi)^{-1}\,d\lm(\xi)
\right\},
\end{split}\end{equation}
\begin{equation}\label{f3RmPgJu}\begin{split}
g_{k\overline{\alpha}}=\frac{\partial^2}{\partial z_k\partial\overline{u_\alpha}}\log K(z,u,z,u)
=\frac{1}{(2\pi)^{2N}K(z,u,z,u)^2}\left((2\pi)^{N}K(z,u,z,u)\right.
\\\cdot\int_{\Omega^*} 2\iu \langle \xi, e_k\rangle\langle \xi, Q(u,f_\alpha)\rangle e^{\iu \langle \xi, z-\overline{z}-2\iu Q(u,u)\rangle}I(\xi)^{-1}I_Q(\xi)^{-1}\,d\lm(\xi)
\\-\left(\int_{\Omega^*}\iu \langle\xi,e_k\rangle e^{\iu \langle\xi,z-\overline{z}-2\iu Q(u,u)\rangle}I(\xi)^{-1}I_Q(\xi)^{-1}\,d\lm(\xi)\right)
\\\quad\cdot\left.\left(\int_{\Omega^*}2\langle\xi,Q(u, f_\alpha)\rangle e^{\iu \langle\xi,z-\overline{z}-2\iu Q(u,u)\rangle}I(\xi)^{-1}I_Q(\xi)^{-1}\,d\lm(\xi)\right)\right),
\end{split}\end{equation}
and 
\begin{equation}\label{cK2nrpCJ}\begin{split}
g_{\alpha\overline{\beta}}=\frac{\partial^2}{\partial u_{\alpha}\partial \overline{u_{\beta}}}\log K(z,u,z,u)=2(2\pi)^{-N}K^{-1}(z,u,z,u)\left(\right.
\\2\int_{\Omega^*}\langle\xi,Q(u,f_{\beta})\rangle\langle\xi,Q(f_\alpha,u)\rangle e^{\iu\langle\xi,z-\overline{z}-2\iu Q(u,u)\rangle}I(\xi)^{-1}I_Q(\xi)^{-1}\,d\lm(\xi)
\\\left.+\int_{\Omega^*}\langle\xi,Q(f_\alpha,f_\beta)\rangle e^{\iu\langle\xi, z-\overline{z}-2\iu Q(u,u)\rangle}I(\xi)^{-1}I_Q(\xi)^{-1}\,d\lm(\xi)\right)
\\-\frac{4}{(2\pi)^{2N}K(z,u,z,u)^2}\int_{\Omega^*}\langle\xi,Q(u,f_\beta)\rangle e^{\iu\langle\xi,z-\overline{z}-2\iu Q(u,u)\rangle}I(\xi)^{-1}I_Q(\xi)^{-1}\,d\lm(\xi)
\\\cdot
\int_{\Omega^*}\langle\xi,Q(f_\alpha,u)\rangle e^{i\langle\xi,z-\overline{z}-2\iu Q(u,u)\rangle}I(\xi)^{-1}I_Q(\xi)^{-1}\,d\lm(\xi).
\end{split}\end{equation}

Now we show Theorem \ref{X6mQfvEy}.
Instead of the original condition (i), we shall consider the following condition:
\begin{equation}
\text{For any }c\in\U,\text{ the representation }\pi_c|_{G^W}\text{ is multiplicity-free. }\tag{i'}
\end{equation}
Note that the proof of (i) $\Rightarrow$ (ii) can be seen from the one of (i') $\Rightarrow$ (ii) below.
For $u\in \U$, we denote by $\overline{u}^W$ the complex conjugation of $u$ with respect to $W\subset\U$.
\begin{proof}[Proof of \textup{(i)} $\Leftrightarrow$ \textup{(ii)}, \textup{(ii)} $\Leftrightarrow$ \textup{(iii)}, \textup{(ii)} $\Leftrightarrow$ \textup{(iv)} of \textup{Theorem \ref{X6mQfvEy}}]

((i') $\Rightarrow$ (ii))
Let $\xi\in\Omega^*$.
Then $\omega_\xi$ is nondegenerate and $N_\xi=\{0\}$.
Owing to Faraut and Thomas \cite[Theorem 2]{FT99}, it follows from Proposition \ref{Thecorresp} that $V_{\xi,0}$ is a multiplicity-free representation of $G^W$.
By Proposition \ref{z8KgCjZK}, we have $W^{\perp,\omega_\xi}\subset W$.
Since $\dim_\mathbb{R} W=\frac{1}{2}\dim_\mathbb{R}\U$, we have also $W\subset W^{\perp,\omega_\xi}$, i.e. $\langle \xi, \mathop{\mathrm{Im}}Q(W,W)\rangle=0$.
Since this holds for every $\xi\in\Omega^*$ and $\Omega^*$ is also a regular cone, we conclude that $\mathop{\mathrm{Im}}Q(W,W)=0$.

((ii) $\Rightarrow$ (i'))
From \eqref{4im} we see that $G^W$ is commutative.
Hence all irreducible invariant Hilbert subspaces are one-dimensional, and it is enough to show that they are of the form
\begin{equation}\label{Sr7Ym8vR}
\mathbb{C} e^{i\langle \xi,z\rangle}e^{h(c,\overline{u}^W)}e^{-\langle\xi,Q(u,\overline{u}^W)\rangle+\frac{1}{2}(h(u,s)-h(s,\overline{u}^W))}\quad(\xi\in (\V)^*, s\in \U).
\end{equation}
For $z=x+\iu y$ with $x,y\in \V$ and $u=w+j w'$ with $w,w'\in W$, we have
\begin{equation*}
n(-x-2Q(w,w'),-w)(z,u)=(\iu(y-Q(w,w)),j w').
\end{equation*}
Let $\mathbb{C} f\subset\mathcal{O}(\mathcal{D}(\Omega,Q))$ be an invariant Hilbert subspace.
Then we have
\begin{equation}\label{M8hqa7NF}\begin{split}
\pi_0(n(x+2Q(w,w'),w))f(z,u)=e^{h(c,w)}f(\iu(y-Q(w,w)),j w')
\\=e^{-\iu\langle\xi,x+2Q(w',w)\rangle+\iu\mathop{\mathrm{Im}}h(s,w)}f(z,u)
\end{split}\end{equation}
for some $\xi\in(\V)^*$ and $s\in\U$.
Note that
\begin{equation}\label{dJ6ytkMi}
f(z,u)=e^{\iu\langle\xi,z\rangle} F(u)
\end{equation}
for some $F\in\mathcal{O}(\U)$.
Combining \eqref{M8hqa7NF} and \eqref{dJ6ytkMi}, we see that
\begin{equation*}
F(u)=e^{h(c,w)}e^{\langle \xi,Q(w,w)\rangle+2\iu\langle\xi,Q(w',w)\rangle-\iu\mathop{\mathrm{Im}}h(s,w)}F(j w'),
\end{equation*}
from which we conclude that the function $F$ and hence $f$ is uniquely determined up to a constant and $\mathbb{C}f$ is as in \eqref{Sr7Ym8vR}.

((ii) $\Rightarrow$ (iii))
First we show that the action of $G^W$ on $\mathcal{D}(\Omega,Q)$ is visible with the totally real submanifold $S:=(\iu \V\times j W)\cap\mathcal{D}(\Omega,Q)$.
It is easy to see that every $G^W$-orbit meets $S$.
For $s:=(\iu y,j w)\in S$, we have
\begin{equation*}
n(x+2Q(w,w'),w')s=(x+\iu y+\iu Q(w',w'),j w+w'),
\end{equation*}
which implies that $J_sT_sS\subset T_s G^Ws$.
Next let $T_sS^\perp$ denote the orthogonal complement of $T_sS$ with respect to the Bergman metric.
We show that $J_sT_s S\subset T_sS^\perp$, which implies $T_sS^\perp\subset J_sT_s S$ and hence (iii), owing to Kobayashi \cite[Theorem 7]{Kobayashi05}.
We see from \eqref{T2xreVk4} that $\mathop{\mathrm{Im}}g_{k\overline{l}}=0\,(1\leq k,l\leq N)$.
The equality \eqref{cK2nrpCJ} tells us that $\mathop{\mathrm{Im}}g_{\alpha\overline{\beta}}=0\,(1\leq \alpha,\beta\leq M)$ under the condition (ii).
Also \eqref{f3RmPgJu} implies that $\mathop{\mathrm{Im}}g_{k\overline{\alpha}}=0\,(1\leq k\leq N, 1\leq \alpha\leq M)$.
We conclude that $J_sT_s S\subset T_sS^\perp$.

((iii) $\Rightarrow$ (ii))
For $u=0$, the assumption together with \eqref{cK2nrpCJ} implies that
\begin{equation*}
\int_{\Omega^*}\langle\xi,\mathop{\mathrm{Im}}Q(f_\alpha,f_\beta)\rangle e^{-2\langle\xi, y\rangle}I(\xi)^{-1}I_Q(\xi)^{-1}\,d\lm(\xi)=0\quad(y\in \Omega, 1\leq\alpha,\beta\leq M).
\end{equation*}
Taking directional derivative in the direction $\mathop{\mathrm{Im}}Q(f_\alpha,f_\beta)\in \V$, we get
\begin{equation*}
\int_{\Omega^*}\langle\xi,\mathop{\mathrm{Im}}Q(f_\alpha,f_\beta)\rangle^2 e^{-2\langle\xi, y\rangle}I(\xi)^{-1}I_Q(\xi)^{-1}\,d\lm(\xi)=0\quad(y\in \Omega, 1\leq \alpha,\beta\leq M),
\end{equation*}
from which we see that $\mathop{\mathrm{Im}}Q(f_\alpha,f_\beta)=0\,(1\leq \alpha,\beta\leq M)$.

((ii) $\Leftrightarrow$ (iv))
For vector-valued functions $a:\mathcal{D}(\Omega,Q)\rightarrow \mathbb{C}^N$ and $b:\mathcal{D}(\Omega,Q)\rightarrow \U$, we shall use the notation
\begin{equation*}
a(z,u)\frac{\partial}{\partial z}+b(z,u)\frac{\partial}{\partial u}:=\sum_{k=1}^Na_k(z,u)\frac{\partial}{\partial z_k}+\sum_{\alpha=1}^M b_\alpha(z,u)\frac{\partial}{\partial u_\alpha}.
\end{equation*}
For $l=1,2,\cdots$, let $\mathbb{D}_{G^{W}}^l(\mathcal{D}(\Omega,Q))\subset \mathbb{D}_{G^{W}}(\mathcal{D}(\Omega,Q))$ denote the space of $l$-th order invariant differential operators.
We first show that
\begin{equation}\label{Fu6bdhTc}
\mathbb{D}_{G^{W}}^1(\mathcal{D}(\Omega,Q))=\left\{\left.\left(2iQ(b,\overline{u}^{W})+A\right)\frac{\partial}{\partial z}+b\frac{\partial}{\partial u}\,\right |\, A\in\mathbb{C}^N, b\in \U\right\}.
\end{equation}
Let $a(z,u)\frac{\partial}{\partial z}+b(z,u)\frac{\partial}{\partial u}\in\mathbb{D}_{G^{W}}(\mathcal{D}(\Omega,Q))$.
Then for $w\in {W}$, we have
\begin{equation*}
a(n(x_0,w)(z,u))=a(z,u)+2iQ(b(z,u),w),
\end{equation*}
and
\begin{equation}\label{NMgL4xdh}
b(n(x_0,w)(z,u))=b(z,u).
\end{equation}
Letting $x_0=-x$ in \eqref{NMgL4xdh}, we see that $b$ does not depend on $z$, since it does not depend on $x$.
Then we can see that $b$ does not depend on $u$ also, i.e. $b\in\U$ is a constant function.
We have
\begin{equation*}\begin{split}
a(z,u)=a(n(x,w)(\iu y-2\iu Q(u,w)+\iu Q(w,w),j w'))
\\=2\iu Q(b,w)+a(\iu y-2\iu Q(u,w)+\iu Q(w,w),j w'),
\end{split}\end{equation*}
which implies that $a$ does not depend on $z$, and we may write $a(z,u)=a(u)=2\iu Q(b,w)+a(j w')$.
By the analytic continuation, we see that $a(u)=2\iu Q(b,\overline{u}^{W})+A$ for some $A\in\mathbb{C}^N$.
Therefore we obtain \eqref{Fu6bdhTc}.

Next, for $b, d\in\U$, we have the following bracket relation of $\mathbb{D}_{G^{W}}(\mathcal{D}(\Omega,Q))$:
\begin{equation*}
\left[2\iu Q(b,\overline{u}^{W})\frac{\partial}{\partial z}+b\frac{\partial}{\partial u}, 2\iu Q(d,\overline{u}^{W})\frac{\partial}{\partial z}+d\frac{\partial}{\partial u}\right]=\{2\iu Q(d,\overline{b}^{W})-2\iu Q(b,\overline{d}^{W})\}\frac{\partial}{\partial z},
\end{equation*}
which shows (iv) $\Rightarrow$ (ii).
Conversely, we show that under (ii) the ring $\mathbb{D}_{G^{W}}(\mathcal{D}(\Omega,Q))$ is generated by $\mathbb{D}_{G^{W}}^1(\mathcal{D}(\Omega,Q))$, and hence (iv) holds.
We show that each $\mathbb{D}_{G^W}^l(\mathcal{D}(\Omega,Q))\,(l=1,2,\cdots)$ is generated by $\mathbb{D}^1(\mathcal{D}(\Omega,Q))$ by induction on $l$.
Let $D\in\mathbb{D}_{G^W}^l(\mathcal{D}(\Omega,Q))$.
Let $\mathbb{Z}_{\geq 0}$ denote the set of nonnegative integers, and for $\boldsymbol{m}=(m_1,m_2,\cdots, m_M)\in\mathbb{Z}_{\geq 0}^M$, we put $|\boldsymbol{m}|:=m_1+m_2+\cdots +m_M$.
Then there exist holomorphic functions $f_{\boldsymbol{m}}\,(\boldsymbol{m}\in\mathbb{Z}_{\geq 0}^M, |\boldsymbol{m}|=l)$ on $\mathcal{D}(\Omega,Q)$ and $l-1$-th order differential operators $D_k\,(k=1,2,\cdots, N)$, generated by $\frac{\partial}{\partial z_{k+1}},\cdots, \frac{\partial}{\partial z_{N}},\frac{\partial}{\partial u_1},\frac{\partial}{\partial u_2},\cdots,\frac{\partial}{\partial u_M}$, with holomorphic coefficients such that
\begin{equation*}\begin{split}
D=\sum_{\begin{array}{c}\boldsymbol{m}\in\mathbb{Z}_{\geq 0}^M\\|\boldsymbol{m}|=l\end{array}}f_{\boldsymbol{m}}(z,u)\prod_{\alpha=1}^M\left(2iQ(e_{\alpha},\overline{u}^{W})\frac{\partial}{\partial z}+\frac{\partial}{\partial u_{\alpha}}\right)^{m_\alpha}+\sum_{k=1}^ND_k\frac{\partial}{\partial z_k}.
\end{split}\end{equation*}
From this expression, it follows that all $f_{\boldsymbol{m}}\,(\boldsymbol{m}\in\mathbb{Z}_{\geq 0}^M, |\boldsymbol{m}|=l)$ are constant functions, and $D_k\,(k=1,2,\cdots, N)$ are $G^{W}$-invariant, and hence $D$ is generated by $\mathbb{D}_{G^{W}}^1(\mathcal{D}(\Omega,Q))$ by assumption.
\end{proof}

\section*{Acknowledgments}
The author would like to thank Professor Ali Baklouti and Professor Hideyuki Ishi for their stimulating discussions on this paper and warm hospitality during the 7th Tunisian-Japanese Conference, Geometric and Harmonic
Analysis on Homogeneous Spaces and Applications in Honor of Professor Toshiyuki Kobayashi in
Monastir, Tunisia, November 1–4, 2023.
He is also grateful to Professor Toshiyuki Kobayashi and Professor Yuichiro Tanaka for valuable discussions, to Professor Atsumu Sasaki for insightful comments.

The author’s work was supported by Foundation of Research Fellows, The Mathematical Society of Japan.


\begin{thebibliography}{99}
\bibitem{Arashi}
K. Arashi, 
Visible actions of certain affine transformation groups of a Siegel domain of the second kind, Lie theory and its applications in physics. Selected papers based on the presentations at the workshop, Sofia, Bulgaria, online, June 2021, ed. by V. Dobrev. Springer Proc. Math. Stat. 396 (Springer, Singapore, 2022).

\bibitem{BS21} A. Baklouti, A. Sasaki, Visible actions and criteria for multiplicity-freeness of representations of Heisenberg groups. J. Lie Theory \textbf{31} (2021), 719--750.

\bibitem{Chen73} S. -S. Chen, Bounded holomorphic functions in Siegel domains, Proc. Amer. Math. Soc. \textbf{40} (1973), 539–-542.

\bibitem{CG88} L. Corwin, F. P. Greenleaf, Spectrum and multiplicities for restrictions of unitary representations in nilpotent Lie groups. Pacific J. Math. \textbf{135} (1988), 233--267.

\bibitem{FT99} J. Faraut, E. G. F. Thomas, Invariant Hilbert spaces of holomorphic functions. J. Lie Theory \textbf{9} (1999), 383–402.

\bibitem{Fujiwara74} H. Fujiwara, On unitary representations of exponential groups, J. Fac. Sci. Tokyo Univ., Sect. I A \textbf{21} (1974), 465--471.

\bibitem{Gindikin} S. G. Gindikin, Analysis in homogeneous domains, Russian Math. Surveys \textbf{19} (1964), 1--89.

\bibitem{Ishi99} H.  Ishi, Representations of the affine transformation groups acting simply transitively on Siegel domains. J. Funct. Anal. \textbf{167} (1999), 425--462.

\bibitem{Ishi11} H. Ishi, Unitary holomorphic multiplier representations over a homogeneous bounded domain. Adv. Pure Appl. Math. \textbf{2} (2011), 405–419.

\bibitem{HW90} A. T. Huckleberry, T. Wurzbacher, Multiplicity-free complex manifolds. Math. Ann. \textbf{286} (1990), 261–-280.

\bibitem{Kobayashi05}
T. Kobayashi, Multiplicity-free representations and visible actions on complex manifolds.
Publ. Res. Inst. Math. Sci. \textbf{41} (2005), 497--549.

\bibitem{Kobayashi07}
T. Kobayashi, Visible actions on symmetric spaces. Transform. Groups \textbf{12} (2007), 671--694.

\bibitem{Kobayashi13}
T. Kobayashi, Propagation of multiplicity-freeness property for holomorphic vector bundles, in Lie groups: structure, actions, and representations. In honor of Joseph A. Wolf on the occasion of his 75th birthday, ed. by A. Huckleberry et al.. Prog. Math. 306 (Birkh\ua user, 2013).

\bibitem{Lisiecki90}
W. Lisiecki, Kaehler coherent state orbits for representations of semisimple Lie groups. Annales de l'IHP. Physique th\'{e}orique \textbf{53} (1990), 245--258.

\bibitem{Lisiecki95}
W. Lisiecki, Coherent state representations. A survey. Rep. Math. Phys. \textbf{35} (1995), 327–358. 

\bibitem{Tanaka22}
Y. Tanaka, Visible actions of compact Lie groups on complex spherical varieties. J. Differential Geom. \textbf{120} (2022), 375--388.

\bibitem{Thomas}
E. G. F. Thomas, The theorem of Bochner-Schwartz-Godement for generalised Gelfand pairs, in Functional analysis: surveys and recent results III, Proc. Conf., 
Paderborn/Ger. 1983, North-Holland Math. Stud. 90 (1984).


\bibitem{VK79}
E. B. Vinberg, B. N. Kimel'fel'd, Homogeneous domains on flag manifolds and spherical subgroups of semisimple Lie groups. Funct. Anal. Appl. \textbf{12} (1979), 168--174.



\end{thebibliography}
\end{document}